\UseRawInputEncoding
\documentclass[reqno]{amsart}
\usepackage{amsmath,amsthm,amscd,amsfonts,amssymb,enumerate}
\usepackage{graphicx}		
\usepackage{color}
\usepackage[colorlinks]{hyperref}
\usepackage{verbatim}
\usepackage{mathrsfs}
\usepackage{bm}

\newtheorem{theorem}{Theorem}[section]

\newtheorem{lemma}[theorem]{Lemma}

\theoremstyle{definition}

\theoremstyle{remark}
\newtheorem{remark}[theorem]{Remark}
\numberwithin{equation}{section}

\newcommand{\upcite}[1]{\textsuperscript{\cite{#1}}}

\newenvironment{proof3.2}{\medskip\noindent{\bf Proof of Theorem 3.2:}\enspace}{\hfill \qed \newline \medskip}
\newenvironment{proof3.3}{\medskip\noindent{\bf Proof of Theorem 3.3:}\enspace}{\hfill \qed \newline \medskip}
\newenvironment{proof4.2}{\medskip\noindent{\bf Proof of Theorem 4.2:}\enspace}{\hfill \qed \newline \medskip}
\begin{document}
\title[On behavior of  solutions to a Petrovsky equation]
{On behavior of  solutions to a Petrovsky equation with damping and variable-exponent source}
The final version of this manuscript has been accepted for publication in SCIENCE CHIAN Mathematics.
\author[M. Liao, Z. Tan ]
{Menglan Liao, Zhong Tan$^*$}

\address{Menglan Liao \newline  
School of Mathematical Sciences, Xiamen University, Xiamen, Fujian, 361005, People's Republic of China }
\email{liaoml14@mails.jlu.edu.cn}
\address{Zhong Tan \newline  
School of Mathematical Sciences, Xiamen University, Xiamen, Fujian, 361005, People's Republic of China }
\email{tan85@xmu.edu.cn}

\thanks{Supported by the National Natural Science Foundation of China (Grant Nos. 11926316, 11531010, 12071391)}
\thanks{$^*$ Corresponding author: Zhong Tan}

\begin{abstract}
This paper deals with the following Petrovsky equation with damping and nonlinear source
\[u_{tt}+\Delta^2 u-M(\|\nabla u\|_2^2)\Delta u-\Delta u_t+|u_t|^{m(x)-2}u_t=|u|^{p(x)-2}u\]
under initial-boundary value conditions, where $M(s)=a+ bs^\gamma$ is a positive $C^1$ function with parameters $a>0,~b>0,~\gamma\geq 1$, and $m(x),~p(x)$ are given measurable functions.  The upper bound of the blow-up time is derived for low initial energy using the differential inequality technique. For $m(x)\equiv2$, in particular, the upper bound of the blow-up time is obtained by the combination of Levine's concavity method and some differential inequalities under high initial energy. In addition, by making full use of the strong damping, the lower bound of the blow-up time is discussed. Moreover, the global existence of solutions and an energy decay estimate are presented by establishing some energy estimates and by exploiting a key integral inequality.

\end{abstract}

\maketitle

\section{Introduction}
It is well known that nonlinear wave equations can be used to describe a variety of problems in physics, engineering, chemistry, material science and other sciences. The study of nonlinear wave equations has also great significance in mathematical analysis. Guesmia \cite{G1998} considered the fourth-order wave equation
\[u_{tt}+\Delta^2u+q(x)u+g(u_t)=0,\]
 for a continuous and increasing function $g$ with $g(0)=0$, and a bounded function $q$􏱃, and proved a global existence and a regularity result. At the same time, decay results for weak and strong solutions  were also established under suitable growth conditions on $g$. Messaoudi \cite{Messaoudi2002} studied the nonlinearly damped semilinear Petrovsky equation
 \[u_{tt}+\Delta^2u+a|u_t|^{m-2}u_t=b|u|^{p-2}u,\]
where $a,􏰼b$ are positive constants.   He proved the existence of a local weak solution and showed that this solution blows up in finite time if $p > m$ with negative initial energy and showed that the solution is global if $m\geq p$. Wu and Tsai \cite{WT2009} extended the  results above, and proved that the solution is global in time under some conditions without the relation between $m$ and $p$. They also proved that the local solution blows up in finite time if $p > m$ and the initial energy is nonngeative. The decay estimates of the energy function and the estimates of the lifespan of solutions were given. The blow-up result has been further improved by Chen and Zhou \cite{CZ2009}, where they proved that the solution blows up in finite time if the positive initial energy satisfies a suitable condition. Li, Sun and Liu \cite{LSL2012} studied the following strongly damped Petrovsky system with nonlinear damping
\[u_{tt}+\Delta^2u-\Delta u_t+|u_t|^{m-1}u_t=b|u|^{p-1}u,\]
where they proved the global existence of the solution under conditions without any relation between $m$ and $p$, and established an exponential decay rate. They also showed that the solution blows up in finite time if $p>m$ and the initial energy is less than the potential well depth. Guo and Li  \cite{GL2019} discussed the lower and upper bounds for the lifespan of solutions to a fourth-order nonlinear hyperbolic equation with strong damping 
\[u_{tt}+\Delta^2u-\Delta u+\omega\Delta u_t+\alpha(t)u_t=b|u|^{p-2}u,\]
and extended and improved the results in \cite{WU2018}.
Readers can refer to \cite{HL2020,TS2012,LG2016, LSW2019,K2019} and the references therein for more other results.

To describe the nonlinear vibrations of an elastic string, Kirchhoff \cite{Kirchhoff} first introduced the following equation:
\[\rho h \frac{\partial^2 u}{\partial t^2}+\delta \frac{\partial u}{\partial t}=\Big\{p_0+\frac{Eh}{2L}\int_0^L \Big(\frac{\partial u}{\partial x}\Big)^2dx\Big\}\frac{\partial^2 u}{\partial x^2}+f,\quad0\leq x\leq L,~t\geq 0,\]
where $u=u(x,t)$ is the lateral deflection, $E$ is the Young's modulus, $\rho$ is the mass density, $h$ is the cross-section area, $L$ is the length, $p_0$ is the initial axial tension, $\delta$ is the resistance modulus, and $f$ is the external force. In recent years,  the problem with Kirchhoff type has been further developed, see \cite{Wu1,Wu2,Yang,Zhou2015}. Zhou  \cite{Zhou2015}, in particular considered a Kirchhoff type plate equation 
\[u_{tt}+\alpha\Delta^2 u-a\Delta u-b \Big(\int_\Omega |\nabla u|^2dx\Big)^\gamma\Delta u+\lambda u_t=\mu |u|^{p-2}u,\]
where they showed the blow-up of solutions and the lifespan estimates for three different ranges of initial energy. Global existence of solutions was proved by the potential well theory, and decay estimates of the energy function were established by using Nakao's inequality.

With the rapid development of the mathematical theory,  much attention has been paid to the study of mathematical nonlinear models of hyperbolic, parabolic and elliptic equations with variable exponents of nonlinearity.  For instance, Messaoudi,   Talahmeh and Al-Smail \cite{SA2017} considered the following  nonlinear wave equation with variable exponents
\[u_{tt}-\Delta u+a|u_t|^{m(x)-2}u_t=b|u|^{p(x)-2}u,\]
they established the existence of a
unique weak solution by using the Faedo-Galerkin method under suitable assumptions, and also proved the finite time blow-up of solutions.

Inspired  by the works mentioned above, in this paper, we are concerned with the following initial-boundary problem:
\begin{equation}
\label{1.1}
\begin{cases}
u_{tt}+\Delta^2 u-M(\|\nabla u\|_2^2)\Delta u-\Delta u_t+|u_t|^{m(x)-2}u_t=|u|^{p(x)-2}u \quad \text{ in }\Omega \times (0,T),\\
 u(x,t)=\frac{\partial u(x,t)}{\partial \nu}=0\quad\text{on } \partial\Omega \times (0,T),\\
      u(x,0)=u_0(x),~u_t(x,0)=u_1(x)\quad\text{in }\Omega,
      \end{cases}
\end{equation}
where $\Omega$ is a bounded smooth domain in $\mathbb{R}^N(N\geq 1)$,  $\nu$ is the unit outward normal vector on $\partial\Omega$,  $u_0(x)\in H_0^2(\Omega)$, $u_1(x)\in L^2(\Omega)$. $M(s)=a+ bs^\gamma$ is a positive $C^1$ function with parameters $a>0,~b>0,~\gamma\geq 1$. The exponents $m(x)$ and $p(x)$ are given measurable functions on $\Omega$ satisfying
 \[m(x)\in [m^-,m^+]\subset (1,\infty),\quad p(x)\in [p^-,p^+]\subset (1,\infty)\quad \forall  x\in \Omega,\]
and the log-H\"{o}lder continuity condition i.e. for some $A>0$ and any $0<\delta<1$:
\begin{equation*}
   |q(x)-q(y)|\leq-\frac{A}{\log|x-y|}\quad \text{for all }x,y\in \Omega~\text{with}~|x-y|<\delta,
\end{equation*}
here 
\[m^-:=ess\inf_{x\in\Omega}m(x),\quad m^+:=ess\sup_{x\in\Omega}m(x),\]\[ p^-:=ess\inf_{x\in\Omega}p(x),\quad p^+:=ess\sup_{x\in\Omega}p(x).\]
To the best of our knowledge, there is no general theory known concerned with existence and nonexistence of solutions to problems like (\ref{1.1}). Antontsev,   Ferreira and  Pi\c{s}kin \cite{AFP2021} obtained the  existence of local weak solutions by using the Banach contraction mapping principle 
under suitable assumptions, and  gave a blow-up result with negative initial energy when $-M(\|\nabla u\|_2^2)\Delta u$ in problem \eqref{1.1} is absent, but they did not discuss the blow-up phenomena with positive initial energy and other properties of solutions.   It seems that one cannot directly apply the classical potential well method to construct some invariant sets and to analyze the behavior of solutions as in \cite{LSL2012,Zhou2015} because for all $f\in L^{p(x)}(\Omega)$, 
\begin{equation*}
    \min\left\{\|f\|_{p(x)}^{p^-},\|f\|_{p(x)}^{p^+}\right\}\leq \int_\Omega |f|^{p(x)}dx\leq\max\left\{\|f\|_{p(x)}^{p^-},\|f\|_{p(x)}^{p^+}\right\},
\end{equation*}
which is different from that $\|f\|_{p}=\Big(\int_\Omega |f|^{p}dx\Big)^{\frac 1p}$ for all $f\in L^p(\Omega)$.

In this paper, we will develop a new technique to discuss bounds of the blow-up time and decay rates. This paper is organized as follows: in Section 2, we introduce the Banach spaces that will be suitable to study problem (\ref{1.1}),  some notations and useful lemmas in the sequel. Sections 3 will be devoted to discussing the lifespan of solutions, i.e. the upper and lower bounds of the blow-up time. In Section 4,  some energy estimates and a key integral inequality \cite{M1999} are used to prove a uniform decay rates of the solution.

\section{Preliminaries}
Throughout this paper,  we denote by $\|\cdot\|_p$ the $L^p(\Omega)$ norm for $1\leq p\le \infty$. We will equip $H_0^2(\Omega)$ with the norm $\|u\|_{H_0^2(\Omega)} = \|\Delta u\|_2$ for $u\in H_0^2(\Omega)$, which is equivalent to the standard one due to Poincar\'e's inequality. Firstly, let us introduce the space $L^{p(x)}(\Omega)$ in \cite{Fan1,Fan2}. Set
\begin{equation}\label{2.1}
    p:\Omega\rightarrow (1,\infty)\text{ be a measurable function.}
\end{equation}
Define
\[A_{p(x)}(f)=\int_\Omega |f|^{p(x)}dx< \infty\]
and
\[L^{p(x)}(\Omega)=\{f\text{ is measurable on }\Omega:A_{p(x)}(f)<\infty\}\]
equipped with the norm
\[\|f\|_{p(x)}=\mathrm{inf}\Big\{\lambda>0:A_{p(x)}\Big(\frac f\lambda\Big)\leq1\Big\}.\]
Let us assume that
\begin{equation}\label{2.2}
    p(x)\in [p^-,p^+]\subset (1,\infty)\quad  a.e.~x\in \Omega.
\end{equation}

\begin{lemma}\label{lem2.1}
Let \eqref{2.1} and \eqref{2.2} be fulfilled.  Then for every $f\in L^{p(x)}(\Omega)$
\begin{equation*}
    \min\left\{\|f\|_{p(x)}^{p^-},\|f\|_{p(x)}^{p^+}\right\}\leq A_{p(x)}(f)\leq\max\left\{\|f\|_{p(x)}^{p^-},\|f\|_{p(x)}^{p^+}\right\}.
\end{equation*}
\end{lemma}

\begin{lemma}
Let $p(x)$ and $q(x)$ satisfy \eqref{2.1} and \eqref{2.2}.  If $p(x)\ge q(x)$ a.e. in $\Omega$, then there is continuous embedding $L^{p(x)}(\Omega)\hookrightarrow L^{q(x)}(\Omega)$ and  the embedding constant is less or equal to $1+|\Omega|$.
\end{lemma}

Let us follow from the proof of Theorem 3.2 in \cite{AFP2021} or Theorem   3.3 in \cite{SA2017}, the local existence of solutions to problem \eqref{1.1} without the proof is as follows.
\begin{theorem}\label{Theo}
Suppose that $u_0(x)\in H_0^2(\Omega)$, $u_1(x)\in L^2(\Omega)$, and 
\begin{equation}
\label{add827}
2\le m^-\leq m(x)\leq m^+<\begin{cases}
\infty&\text{ for }N\le  4;\\
2N/(N-4)&\text{ for }N\ge 5;
\end{cases}
\end{equation}
\begin{equation}
\label{1add827}
2<p^-\leq p(x)\leq p^+<\begin{cases}
\infty&\text{ for }N\le  4;\\
2(N-2)/(N-4)&\text{ for }N\ge 5;
\end{cases}
\end{equation}
then there exists a unique local weak solution $u:=u(x,t)\in L^\infty(0,T;H_0^2(\Omega))$ for problem $(\ref{1.1})$ with
\[u_t\in L^\infty(0,T;L^2(\Omega))\cap L^{m^-}(0,T;L^{m(x)}(\Omega))\cap L^2(0,T;H_0^1(\Omega)).\]
\end{theorem}

Define in the sequel 
\[\alpha_1:=(B_1^2)^{\frac{-2}{p^--2}},\quad E_1:=\Big(\frac12-\frac{1}{p^-}\Big)\alpha_1^{\frac{p^-}{2}}\]
with $B_1=\max\big\{1,B\big\}$, here $B$ is the embedding constant to the embedding $H_0^2(\Omega)\hookrightarrow L^{p(x)}(\Omega)$, i.e. there exists an optimal constant $B$ such that
\begin{equation}
\label{Eq3}
\|u\|_{p(x)}\leq B\|\Delta u\|_2\quad \forall u\in H_0^2(\Omega).
\end{equation}

Define the energy functional
\begin{equation}\label{3.1}
\begin{split}
E(t)=\frac{1}{2}\|u_t\|_{2}^{2}+\frac{1}{2}\|\Delta u\|_2^2+\frac{a}{2}\|\nabla u\|_2^2+\frac{b}{2(\gamma+1)}\|\nabla u\|_2^{2(\gamma+1)}-\int_\Omega\frac{1}{p(x)}|u|^{p(x)}dx.
\end{split}
\end{equation}
A direct computation implies 
\begin{equation}\label{3.2}
E'(t)=-\int_\Omega|u_t|^{m(x)}dx-\|\nabla u_t\|_2^2\leq 0.
\end{equation}
\begin{lemma}\label{lem3.2}
Let $E(0)<E_1$ and 
\begin{equation}
\label{2add827}
2<p^-\leq p(x)\leq p^+<\begin{cases}
\infty&\text{ for }N\le  4;\\
2N/(N-4)&\text{ for }N\ge 5.
\end{cases}
\end{equation}
Assume that $u$ is a solution for problem $(\ref{1.1})$, then
\begin{enumerate}[$(1)$]
  \item for  $B_1^2\|\Delta u_0\|^2_2>\alpha_1$, there exists a positive constant $\alpha_2>\alpha_1$ such that
\begin{equation}\label{3.5}
B_1^2\|\Delta u\|_2^2\geq\alpha_2\quad\forall t\ge0.
\end{equation}
  \item for   $B_1^2\|\Delta u_0\|^2_2<\alpha_1$, there exists a positive constant $0<\tilde{\alpha_2}<\alpha_1$ such that
\begin{equation}\label{5.1}
B_1^2\|\Delta u\|^2_2\le\tilde{\alpha_2}\quad\forall t\ge 0.
\end{equation}
\end{enumerate}
\end{lemma}
\begin{proof}
Using (\ref{3.1}), Lemma \ref{lem2.1} and (\ref{Eq3}) yields  that
\begin{equation}\label{3.6}
\begin{split}
E(t)&\geq \frac{1}{2}\|\Delta u\|_2^2-\frac{1}{p^-}\max\Big\{\|u\|_{p(x)}^{p^+},\|u\|_{p(x)}^{p^-}\Big\}\\
&\geq\frac{1}{2}\|\Delta u\|_2^2-\frac{1}{p^-}\max\Big\{B^{p^+}\|\Delta u\|_2^{p^+},B^{p^-}\|\Delta u\|_2^{p^-}\Big\}\\
&\geq\frac{1}{2B_1^2}\alpha-\frac{1}{p^-}\max\Big\{\alpha^{\frac{p^+}{2}},\alpha^{\frac{p^-}{2}}\Big\}:=G(\alpha),
\end{split}
\end{equation}
where $\alpha:=\alpha(t)=B_1^2\|\Delta u\|_2^2$. By direct computations, $G(\alpha)$ satisfies the following:
\begin{align*}
G'(\alpha)&= \begin{cases}
\frac{1}{2B_1^2}-\frac{p^+}{2p^-}\alpha^{\frac{p^+-2}{2}}<0&\alpha>1;\\
\frac{1}{2B_1^2}-\frac{1}{2}\alpha^{\frac{p^--2}{2}}&0<\alpha<1;
\end{cases} \\
G'_+(1)&=\frac{1}{2B_1^2}-\frac{p^+}{2p^-}<0,~G'_-(1) =\frac{1}{2B_1^2}-\frac{1}{2}<0;\\
G'(\alpha_1)&=0,~0<\alpha_1<1.
\end{align*}
Thus, $G(\alpha)$ is strictly increasing for $0<\alpha<\alpha_1$,  strictly decreasing for $\alpha_1<\alpha$, $G(\alpha)\rightarrow -\infty$ as  $\alpha\rightarrow +\infty$, and $G(\alpha_1)=E_1$.
Since $E(0)<E_1$, there exist  $\alpha_2$ and $\tilde{\alpha_2}$ with $\alpha_2>\alpha_1>\tilde{\alpha_2}$  such that $G(\alpha_2)=G(\tilde{\alpha_2})=E(0)$. Set $\alpha_0:=B_1^2\|\Delta u_0\|_2^2$,  then 
\begin{equation}
\label{2add817}
G(\alpha_0)\leq E(0)=G(\alpha_2)=G(\tilde{\alpha_2}).
\end{equation}
\begin{enumerate}[$(1)$]
  \item  if  $B_1^2\|\Delta u_0\|^2_2>\alpha_1$, then \eqref{2add817} implies 
$\alpha_0\geq \alpha_2$. To prove $(\ref{3.5})$, we suppose by contradiction that for some $t_0>0$,
$\alpha(t_0)<\alpha_2.$ The continuity of $\alpha(t)$ illustrates that we could choose $t_0$
such that $\alpha_1<\alpha(t_0)<\alpha_2$, we have
$E(0)=G(\alpha_2)<G(\alpha(t_0))\leq E(t_0),$  which contradicts (\ref{3.2}).
  \item  if  $B_1^2\|\Delta u_0\|^2_2<\alpha_1$, then \eqref{2add817} implies 
$\alpha_0\le\tilde{\alpha_2}$. Similar to (1), we suppose by contradiction that for some $t^0>0$,
$\alpha(t^0)>\tilde{\alpha_2}.$ The continuity of $\alpha(t)$ illustrates that we could choose $t^0$
such that $\tilde{\alpha_2}<\alpha(t^0)<\alpha_1$, we have
$E(0)=G(\tilde{\alpha_2})<G(\alpha(t^0))\leq E(t^0),$  which contradicts (\ref{3.2}).
\end{enumerate}
\end{proof}

\begin{lemma}\label{lem3.3}
Set $H(t)=E_2-E(t)$ for $t\geq 0$, where $E_2\in(E(0),E_1)$ is sufficiently close to $E(0)$.
 If all conditions of Lemma $\ref{lem3.2}(1)$ hold, for all $t\geq0$,
\begin{equation}\label{3.7}
0<H(0)\leq H(t)\leq \int_\Omega\frac{1}{p(x)}|u|^{p(x)}dx\leq \frac{1}{p^-}\int_\Omega|u|^{p(x)}dx.
\end{equation}
\end{lemma}
\begin{proof}
(\ref{3.2}) implies that $H(t)$ is nondecreasing with respect to $t$, thus for  $t\geq 0$, $H(t)\geq H(0)=E_2-E(0)>0$. (\ref{3.1}) and (\ref{3.5}) illustrate
 \begin{equation*}
    \begin{split}
      H(t)&\leq E_2-\frac{1}{2}\|\Delta u\|_2^2+\int_\Omega\frac{1}{p(x)}|u|^{p(x)}dx \leq E_2-\frac{\alpha_2}{2B_1^2}+\int_\Omega\frac{1}{p(x)}|u|^{p(x)}dx\\
& \leq E_1-\frac{\alpha_1}{2B_1^2}+\int_\Omega\frac{1}{p(x)}|u|^{p(x)}dx
\leq  \int_\Omega\frac{1}{p(x)}|u|^{p(x)}dx.
    \end{split}
\end{equation*}
\end{proof}

\begin{remark}
In Lemma \ref{lem3.3},  let $H(t)=E_2-E(t)$, where $E_2\in(E(0),E_1)$ is sufficiently close to $E(0)$ be necessary. This necessity can be seen in the proof of Theorem \ref{theo3.1}.
\end{remark}

\begin{lemma}\label{lem2.6}
Let $p^->2(\gamma+1)$  and $m(x)=2$ hold. If the initial datum $u_0\in H_0^2(\Omega)$ and $u_1\in L^2(\Omega)$ such that 
\begin{equation}
\label{61}
0<E(0)<\frac{C}{p^-}\int_\Omega u_0u_1dx.
\end{equation}
Then the weak solution $u$ to problem \eqref{1.1} satisfies 
\begin{equation}
\label{62}
\int_\Omega uu_tdx-\frac{p^-}{C}E(t)\ge \Big(\int_\Omega u_0u_1dx-\frac{p^-}{C}E(0)\Big)e^{Ct}>0
\end{equation}
for any $t\in [0,T),$ where
\begin{equation}
\label{63}
C=\min\Big\{2+p^-,\frac{2p^-(p^--2)a}{1+(2p^-+1)S^2}\Big\}
\end{equation}
with $S$ being the optimal embedding constant of $H_0^1(\Omega)\hookrightarrow L^{2}(\Omega)$.
\end{lemma}
\begin{proof}
The idea of this proof comes from \cite{SLW2019,HL2020,L2021}. It is direct that \begin{equation}
\label{64}
\begin{split}
&\frac{d}{dt}\int_\Omega uu_tdx=\|u_t\|_2^2+\int_\Omega uu_{tt}dx\\
&\quad=\|u_t\|_2^2-\|\Delta u\|_2^2-a\|\nabla u\|_2^2-b\|\nabla u\|_2^{2(\gamma+1)}\\
&\quad\quad-\int_\Omega\nabla u\nabla u_tdx-\int_\Omega u u_tdx+\int_\Omega |u|^{p(x)}dx
\\
&\quad\ge \frac{p^-+2}{2}\|u_t\|_2^2+\frac{p^--2}{2}\|\Delta u\|_2^2+\frac{p^--2}{2}a\|\nabla u\|_2^2\\
&\quad\quad+\frac{p^--2(\gamma+1)}{2(\gamma+1)}b\|\nabla u\|_2^{2(\gamma+1)}-\int_\Omega\nabla u\nabla u_tdx-\int_\Omega u u_tdx-p^-E(t),
\end{split}
\end{equation}
by using the first equality for problem \eqref{1.1} and \eqref{3.1}.  
Taking full advantage of Young's inequality, then
\begin{equation}
\label{65}
\int_\Omega\nabla u\nabla u_tdx\le \frac{C}{4p^-}\|\nabla u\|_2^2+\frac {p^-}{C}\|\nabla u_t\|_2^2.
\end{equation}
\begin{equation}
\label{66}
\int_\Omega u u_tdx\le \frac{C}{4p^-}\| u\|_2^2+\frac {p^-}{C}\| u_t\|_2^2.
\end{equation}
Combining \eqref{64} and \eqref{65} with \eqref{66}, and using the embedding theorem $H_0^1(\Omega)\hookrightarrow L^{2}(\Omega)$, it follows that 
\begin{equation}
\label{67}
\begin{split}
\frac{d}{dt}\int_\Omega uu_tdx&\ge\frac{p^-+2}{2}\|u_t\|_2^2+\Big[\frac{p^--2}{2}a-\frac{C}{4p^-}\Big]\|\nabla u\|_2^2-\frac{C}{4p^-}\| u\|_2^2\\
&\quad+\frac {p^-}{C}\Big(\|\nabla u_t\|_2^2+\|u_t\|_2^2\Big)-p^-E(t)\\
&\ge\frac{p^-+2}{2}\|u_t\|_2^2+\Big\{\Big[\frac{p^--2}{2}a-\frac{C}{4p^-}\Big]\frac{1}{S^2}-\frac{C}{4p^-}\Big\}\| u\|_2^2\\
&\quad+\frac {p^-}{C}\Big(\|\nabla u_t\|_2^2+\|u_t\|_2^2\Big)-p^-E(t).
\end{split}
\end{equation}Define 
\[\Psi(t)=\int_\Omega uu_tdx-\frac {p^-}{C}E(t).\] 
Recalling \eqref{63}, and then combining \eqref{67} with \eqref{3.2}, one has
\begin{equation}
\label{68}
\begin{split}
\frac{d}{dt}\Psi(t)&\ge \frac{p^-+2}{2}\|u_t\|_2^2+\Big\{\Big[\frac{p^--2}{2}a-\frac{C}{4p^-}\Big]\frac{1}{S^2}-\frac{C}{4p^-}\Big\}\| u\|_2^2-p^-E(t)\\
&\ge C\Big(\frac12\|u_t\|_2^2+\frac12\|u\|_2^2-\frac {p^-}{C}E(t)\Big)\ge C\Psi(t).
\end{split}
\end{equation}
Noticing that $\Psi(0)=\int_\Omega u_0u_1dx-\frac {p^-}{C}E(0)>0$,  by Gronwall's inequality, thus 
\[\Psi(t)\ge e^{Ct}\Psi(0)>0.\]
This proof is complete.
\end{proof}

\begin{lemma}\label{lem2.5}
\upcite{L1973,LG2017}Suppose a positive, twice-differentiable function $\psi (t)$ satisfies the inequality
\begin{equation*}
\psi''(t)\psi(t)-(1+\theta)({\psi'}(t))^2\ge 0,
\end{equation*}
where $\theta>0.$ If $\psi(0)>0,~\psi'(0)>0$, then $\psi(t)\to\infty$ as $t\to t_1\le t_2=\frac{\psi(0)}{\theta\psi'(0)}$.
\end{lemma}
\begin{lemma}\label{lem5.4}
\upcite{M1999}Let $E: \mathbb{R}\to\mathbb{R}_+$ be a non-increasing function and $\phi : \mathbb{R}\to\mathbb{R}_+$ be a strictly increasing function of class $C^1$ such that
$\phi(0)=0\text{ and }\phi(t)\to+\infty\text{ as }t\to+\infty.$ 
Assume that there exist $\sigma \geq  0$ and $\omega > 0$ such that:
 \[\int_t^{+\infty}(E(s))^{1+\sigma}\phi^{\prime}(s)ds\leq  \frac{1}{\omega}(E(0))^\sigma E(t),\]
 then $E$ has the following decay property:
\begin{enumerate}[$(1)$]
  \item if $\sigma=0$, then $E(t)\leq  E(0)e^{1-\omega\phi(t)}$ for all $t\geq  0;$
  \item if $\sigma>0$, then $E(t)\leq  E(0)\left(\frac{1+\sigma}{1+\omega\sigma\phi(t)}\right)^{\frac{1}{\sigma}}$ for all $t\geq  0$.
\end{enumerate}
\end{lemma}

\section{Blow-up results}
In this section, the blow-up phenomenon will be discusse. Moreover, the lifespan of solutions will be derived as well.
\subsection{Blow-up for low initial energy}
By constructing an ordinary differential inequality, let us present the blow-up result for low initial energy as follows: 
\begin{theorem}\label{theo3.1}
Let \eqref{2add827} and $\max\{m^+,2(\gamma+1)\}<p^-$ hold. Provided that 
\[B_1^2\|\Delta u_0\|^2_2>\alpha_1\text{ and }E(0)<E_1,\]
 then the solution $u$ of problem $(\ref{1.1})$ blows up at some  finite time $T^*$ in the sense of  $\lim\limits_{t\rightarrow T^{*-}}\|u\|_{p(x)}=+\infty$ and  the blow-up time $T^*$ can be estimated from above as follows
\[T^*\leq F^{-\frac{\sigma}{1-\sigma}}(0)\frac{M_1}{M_2}\frac{1-\sigma}{\sigma},\]
where \[0<\sigma\leq\min\left\{\frac{p^--m^+}{p^-(m^+-1)},\frac{p^--2}{2p^-},\frac{\gamma}{\gamma+1}\right\}<\frac12,\] 
$M_1$ and $M_2$ will be presented in \eqref{3add827} and \eqref{4add827}, respectively. 
\end{theorem}
\begin{proof}

This proof goes back to our previous paper \cite{LGZ2020}. Let us recall $H(t)=E_2-E(t)$ for $t\geq 0$. Define an auxiliary function
\[F(t)=H^{1-\sigma}(t)+\varepsilon\Big(\int_\Omega u_tudx+\frac12\|\nabla u\|_2^2\Big).\]
The proof will be divided into three steps:

\textbf{Step 1: Estimate for \bm{$F'(t)$}} Differentiating directly $F(t)$, recalling (\ref{1.1}), adding and subtracting $\varepsilon p^-(1-\varepsilon_1)H(t)$ with $0<\varepsilon_1<1$ is to get
\begin{equation}\label{3.8}
    \begin{split}
     F'(t) &
        = (1-\sigma)H^{-\sigma}(t)H'(t)+\varepsilon\|u_t\|_2^2-\varepsilon\|\Delta u\|_2^2-\varepsilon a\|\nabla u\|_2^2\\
        &\quad-\varepsilon b \|\nabla u\|_2^{2(\gamma+1)}
        -\varepsilon \int_\Omega|u_t|^{m(x)-2}u_tudx+\varepsilon \int_\Omega|u|^{p(x)}dx\\
        &\geq (1-\sigma)H^{-\sigma}(t)H'(t)+\varepsilon\Big[1+\frac{p^-(1-\varepsilon_1)}{2}\Big]\|u_t\|_2^2\\
        &\quad+\varepsilon \Big[\frac{p^-(1-\varepsilon_1)}{2}-1\Big]\|\Delta u\|_2^2+\varepsilon a \Big[\frac{p^-(1-\varepsilon_1)}{2}-1\Big]\|\nabla u\|_2^2\\
       & \quad+\varepsilon b\Big[\frac{p^-(1-\varepsilon_1)}{2(r+1)}-1\Big] \|\nabla u\|_2^{2(\gamma+1)}-\varepsilon \int_\Omega|u_t|^{m(x)-2}u_tudx\\
        &\quad+\varepsilon p^-(1-\varepsilon_1)H(t)-\varepsilon p^-(1-\varepsilon_1)E_2+\varepsilon \varepsilon_1\int_\Omega|u|^{p(x)}dx.
    \end{split}
\end{equation}
Applying Young's inequality with $\varepsilon_2>1$ and $H'(t)=-E'(t)$, Lemma \ref{lem2.1}, the embedding $L^{p(x)}(\Omega)\hookrightarrow L^{m(x)}(\Omega)$,  we easily have 
\begin{equation}\label{3.9}
    \begin{split}
    & \Big|\int_\Omega|u_t|^{m(x)-2}u_tudx\leq\varepsilon_2H^{-\sigma}(t)\int_\Omega|u_t|^{m(x)}dx+\frac{1}{\varepsilon_2^{m^--1}}\int_\Omega |u|^{m(x)}H^{\sigma[m(x)-1]}(t)dx\\
      &\quad\leq\varepsilon_2H^{-\sigma}(t)\int_\Omega|u_t|^{m(x)}dx+\frac{C_1^{\sigma(m^--m^+)}H^{\sigma(m^+-1)}(t)}{\varepsilon_2^{m^--1}}\int_\Omega |u|^{m(x)}dx\\
      &\quad\leq\varepsilon_2H^{-\sigma}(t)\int_\Omega|u_t|^{m(x)}dx+\frac{C_2H^{\sigma(m^+-1)}(t)}{\varepsilon_2^{m^--1}}
      \max\Big\{\|u\|_{p(x)}^{m^+},\|u\|_{p(x)}^{m^-}\Big\}\\
     &\quad\leq \varepsilon_2H^{-\sigma}(t)H'(t)+\frac{C_2H^{\sigma(m^+-1)}(t)}{\varepsilon_2^{m^--1}}
      \max\Big\{\|u\|_{p(x)}^{m^+},\|u\|_{p(x)}^{m^-}\Big\} ,
    \end{split}
\end{equation}
where
$C_1=\min\{H(0),1\},~C_2=(1+|\Omega|)^{m^+}C_1^{\sigma(m^--m^+)}.$
On the other hand, Lemma \ref{lem2.1} and Lemma \ref{lem3.3} imply 
\begin{equation*}\label{3.10}
    \begin{split}
     \|u\|_{p(x)}^{m^+}&\leq  \max\Big\{\Big(\int_\Omega|u|^{p(x)}dx\Big)^{\frac{m^+}{p^+}},\Big(\int_\Omega|u|^{p(x)}dx\Big)^{\frac{m^+}{p^-}}\Big\}\\
     &\leq \max\Big\{\Big(p^-H(t)\Big)^{\frac{m^+}{p^+}-\frac{m^+}{p^-}},1\Big\}\Big(\int_\Omega|u|^{p(x)}dx\Big)^{\frac{m^+}{p^-}};
    \end{split}
\end{equation*}
\begin{equation*}\label{3.11}
    \begin{split}
     \|u\|_{p(x)}^{m^-}
     \leq \max\Big\{\Big( p^-H(t) \Big)^{\frac{m^-}{p^+}-\frac{m^+}{p^-}},\Big( p^-H(t) \Big)^{\frac{m^--m^+}{p^-}}\Big\}\Big(\int_\Omega|u|^{p(x)}dx\Big)^{\frac{m^+}{p^-}},
    \end{split}
\end{equation*}
which illustrate
\begin{equation}\label{3.12}
   \max\left\{\|u\|_{p(x)}^{m^+},\|u\|_{p(x)}^{m^-}\right\}\leq C_3\Big(\int_\Omega|u|^{p(x)}dx\Big)^{\frac{m^+}{p^-}},
\end{equation}
where $C_3=2\Big(\min\Big\{p^-H(0),1\Big\}\Big)^{\frac{m^-}{p^+}-\frac{m^+}{p^-}}.$ Recalling $0<\sigma\leq \frac{p^--m^+}{p^-(m^+-1)}$ and Lemma \ref{lem3.3}, apparently, 
\begin{equation}\label{addd3.13}
    \begin{split}
    &H^{\sigma(m^+-1)}(t) \max\Big\{\|u\|_{p(x)}^{m^+},\|u\|_{p(x)}^{m^-}\Big\} \leq C_3H^{\sigma(m^+-1)}(t)\Big(\int_\Omega|u|^{p(x)}dx\Big)^{\frac{m^+}{p^-}}\\
      &\quad= C_3\frac{H^{\sigma(m^+-1)+\frac{m^+}{p^-}-1}(t)}{H^{\sigma(m^+-1)+\frac{m^+}{p^-}-1}(0)}H^{1-\frac{m^+}{p^-}}(t)H^{\sigma(m^+-1)+\frac{m^+}{p^-}-1}(0)
     \Big(\int_\Omega|u|^{p(x)}dx\Big)^{\frac{m^+}{p^-}}\\
      &\quad\leq C_3\Big(\frac{1}{p^-}\Big)^{1-\frac{m^+}{p^-}}\Big(\int_\Omega|u|^{p(x)}dx\Big)^{1-\frac{m^+}{p^-}}H^{\sigma(m^+-1)+\frac{m^+}{p^-}-1}(0)
     \Big(\int_\Omega|u|^{p(x)}dx\Big)^{\frac{m^+}{p^-}}\\
      &\quad\leq C_3\Big(\frac{1}{p^-}\Big)^{1-\frac{m^+}{p^-}}C_1^{\sigma(m^+-1)+\frac{m^+}{p^-}-1}\int_\Omega|u|^{p(x)}dx.
    \end{split}
\end{equation}
It follows from  (\ref{3.8}), (\ref{3.9}) and (\ref{addd3.13}) that
\begin{equation}\label{3.14}
    \begin{split}
      F'(t)
       &\geq(1-\sigma-\varepsilon\varepsilon_2)H^{-\sigma}(t)H'(t)+\varepsilon\Big[1+\frac{p^-(1-\varepsilon_1)}{2}\Big]\|u_t\|_2^2\\
        &\quad+\varepsilon \Big[\frac{p^-(1-\varepsilon_1)}{2}-1\Big]\|\Delta u\|_2^2+\varepsilon p^-(1-\varepsilon_1)H(t)-\varepsilon p^-(1-\varepsilon_1)E_2\\
        &\quad+\varepsilon a \Big[\frac{p^-(1-\varepsilon_1)}{2}-1\Big]\|\nabla u\|_2^2+\varepsilon b\Big[\frac{p^-(1-\varepsilon_1)}{2(r+1)}-1\Big]\|\nabla u\|_2^{2(\gamma+1)}\\
        &\quad+\varepsilon\Big[\varepsilon_1-\frac{  C_1^{\sigma(m^+-1)+\frac{m^+}{p^-}-1}C_2C_3\Big(\frac{1}{p^-}\Big)^{1-\frac{m^+}{p^-}}}{\varepsilon_2^{m^--1}}\Big]\int_\Omega|u|^{p(x)}dx.
    \end{split}
\end{equation}
Let us choose $0<\varepsilon_1<\frac{p^--2(\gamma+1)}{p^-}<\frac{p^--2}{p^-}$ sufficiently small and choose  $E_2\in(E(0),E_1)$, sufficiently close to $E(0)$ such that 
\begin{equation*}
    E_2\leq \Big[\frac12-\frac{1}{p^-(1-\varepsilon_1)}\Big]\alpha_1^{\frac{p^-}{2}}<E_1,
\end{equation*}
therefore, Lemma \ref{lem3.2}(1) implies 
\begin{equation*}
\begin{split}
&\varepsilon \Big[\frac{p^-(1-\varepsilon_1)}{2}-1\Big]\|\Delta u\|_2^2-\varepsilon p^-(1-\varepsilon_1)E_2\geq\varepsilon \Big[\frac{p^-(1-\varepsilon_1)}{2}-1\Big]\frac{\alpha_2}{B_1^2}-\varepsilon p^-(1-\varepsilon_1)E_2\\
&\quad\geq \varepsilon \Big[\frac{p^-(1-\varepsilon_1)}{2}-1\Big]\alpha_1^{\frac{p^-}{2}}-\varepsilon p^-(1-\varepsilon_1)E_2\geq 0.
\end{split}
\end{equation*}
Let us fix the constant $\varepsilon_2$ such that
\[\varepsilon_1>\frac{C_1^{\sigma(m^+-1)+\frac{m^+}{p^-}-1}C_2C_3\Big(\frac{1}{p^-}\Big)^{1-\frac{m^+}{p^-}}}{\varepsilon_2^{m^--1}},\] 
and then choose $\varepsilon$ so small that
$1-\sigma>\varepsilon\varepsilon_2.$ Therefore, (\ref{3.14}) can be written as 
\begin{equation}\label{3.15} 
      F'(t)
       \geq M_1\Big(\|u_t\|_{2}^{2}+ H(t)+\|\nabla u\|_2^{2(\gamma+1)}+\|\nabla u\|_2^2+\int_\Omega|u|^{p(x)}dx\Big),
\end{equation}
where
\begin{equation}
\label{3add827}
\begin{split}
M_1&=\varepsilon\min\Big\{1+\frac{p^-(1-\varepsilon_1)}{2},(1-\varepsilon_1)p^-,a \Big[\frac{p^-(1-\varepsilon_1)}{2}-1\Big],\\
&\quad b\Big[\frac{p^-(1-\varepsilon_1)}{2(r+1)}-1\Big],
\Big[\varepsilon_1-\frac{C_1^{\sigma(m^+-1)+\frac{m^+}{p^-}-1}C_2C_3\Big(\frac{1}{p^-}\Big)^{1-\frac{m^+}{p^-}}}{\varepsilon_2^{m^--1}}\Big]\Big\}.
\end{split}
\end{equation}

\textbf{Step 2: Estimate for \bm{$F^{\frac{1}{1-\sigma}}(t)$}}  We are  now in a position to consider
\begin{equation}
\label{add3.16}
F^{\frac{1}{1-\sigma}}(t)=\Big[H^{1-\sigma}(t)+\varepsilon\Big(\int_\Omega u_tudx+\frac12\|\nabla u\|_2^2\Big)\Big]^{\frac{1}{1-\sigma}}.
\end{equation}
On the one hand, applying H\"{o}lder's inequality, embedding $L^{p(x)}(\Omega)\hookrightarrow L^{2}(\Omega)$ and Young's inequality shows 
\begin{equation}\label{3.16}
    \begin{split}
     \Big|\int_\Omega u_tudx\Big|^{\frac{1}{1-\sigma}}
       &\leq\left(\|u_t\|_{2}\|u\|_{2}\right)^{\frac{1}{1-\sigma}}\leq (1+|\Omega|)^{\frac{1}{1-\sigma}}\|u_t\|_{2}^{\frac{1}{1-\sigma}}\|u\|_{p(x)}^{\frac{1}{1-\sigma}}\\
       &\leq C_4\|u_t\|_{2}^{2}+C_5\|u\|_{p(x)}^{\frac{2}{2(1-\sigma)-1}},
    \end{split}
\end{equation}
where \[C_4=\frac{(1+|\Omega|)^{\frac{1}{1-\sigma}}}{2(1-\sigma)}, ~C_5=\frac{(1+|\Omega|)^{\frac{1}{1-\sigma}}[2(1-\sigma)-1]}{2(1-\sigma)}.\] 
Recalling $0<\sigma\leq\frac{p^--2}{2p^-}$, Lemma \ref{lem2.1} and Lemma \ref{lem3.3},  one obtains
 \begin{equation}\label{3.17}
    \begin{split}
    &\|u\|_{p(x)}^{\frac{2}{2(1-\sigma)-1}}\leq\max\Big\{\Big(\int_\Omega|u|^{p(x)}dx\Big)^{\frac{2}{p^-[2(1-\sigma)-1]}},
        \Big(\int_\Omega|u|^{p(x)}dx\Big)^{\frac{2}{p^+[2(1-\sigma)-1]}}\Big\}\\
          &\leq\max\Big\{\Big(p^-H(t)\Big)^{\frac{2-p^-[2(1-\sigma)-1]}{p^-[2(1-\sigma)-1]}},
        \Big(p^-H(t)\Big)^{\frac{2-p^+[2(1-\sigma)-1]}{p^+[2(1-\sigma)-1]}}\Big\}\int_\Omega|u|^{p(x)}dx\\
        &\leq C_6\int_\Omega|u|^{p(x)}dx
    \end{split}
\end{equation}
with $C_6=\Big(\min\Big\{p^-H(0),1\Big\}\Big)^{\frac{2-p^+[2(1-\sigma)-1]}{p^+[2(1-\sigma)-1]}}.$ Inserting (\ref{3.17}) into (\ref{3.16}) yields 
\begin{equation}
\label{3.18}
    \Big|\int_\Omega u_tudx\Big|^{\frac{1}{1-\sigma}}
      \leq C_4\|u_t\|_{2}^{2}
      +C_5C_6\int_\Omega|u|^{p(x)}dx.
\end{equation}
On the other hand, there exists a positive constant $C_7$ such that
\begin{equation}
\label{3.19}
\|\nabla u\|_2^{\frac{2}{1-\sigma}}\leq C_7(\|\nabla u\|_2^2+\|\nabla u\|_2^{2(\gamma+1)}).
\end{equation}
Let us pause to prove this inequality. If $\|\nabla u\|_2<1$, it directly follows from the fact $\frac{2}{1-\sigma}>2$ that $\|\nabla u\|_2^{\frac{2}{1-\sigma}}\leq \|\nabla u\|_2^2$. If $\|\nabla u\|_2\geq 1$, recalling $\sigma<\frac{\gamma}{\gamma+1}$, we have $\frac{2}{1-\sigma} < 2(\gamma+1)$,  which implies $\|\nabla u\|_2^{\frac{2}{1-\sigma}}\leq \|\nabla u\|_2^{2(\gamma+1)}$.
Therefore, combining (\ref{3.17}) (\ref{3.19})  (\ref{add3.16}) with 
\begin{equation}
\label{add3.20}
(a_1+a_2+\cdots+a_m)^l\leq 2^{(m-1)(l-1)}(a_1^l+a_2^l+\cdots+a_m^l),\end{equation}
here $a_i\geq 0(i=1,2,\cdots,m),~l\geq 1,~m\geq1,$ we get
\begin{equation}\label{adddd3.20}
   \begin{split}
      &F^{\frac{1}{1-\sigma}}(t)
       \leq 2^{\frac{2\sigma}{1-\sigma}}\Big(H(t)+\varepsilon^{\frac{1}{1-\sigma}}\Big|\int_\Omega u_tudx\Big|^{\frac{1}{1-\sigma}}+\Big(\frac12\Big)^{\frac{1}{1-\sigma}}\|\nabla u\|_2^{\frac{2}{1-\sigma}}\Big)\\
       &\leq  2^{\frac{2\sigma}{1-\sigma}}\Big(H(t)+\varepsilon^{\frac{1}{1-\sigma}}C_4\|u_t\|_{2}^{2}
      +\varepsilon^{\frac{1}{1-\sigma}}C_5C_6\int_\Omega|u|^{p(x)}dx\\
      &\quad+\Big(\frac12\Big)^{\frac{1}{1-\sigma}}C_7\|\nabla u\|_2^2+\Big(\frac12\Big)^{\frac{1}{1-\sigma}}C_7\|\nabla u\|_2^{2(\gamma+1)}\Big)\\
      &\leq M_2\Big(H(t)+\|u_t\|_{2}^{2}+\|\nabla u\|_2^2+\|\nabla u\|_2^{2(\gamma+1)}+\int_\Omega|u|^{p(x)}dx\Big),
   \end{split}
\end{equation}
where 
\begin{equation}
\label{4add827}
M_2=2^{\frac{2\sigma}{1-\sigma}}\max\left\{1,\varepsilon^{\frac{1}{1-\sigma}}C_4,
\varepsilon^{\frac{1}{1-\sigma}}C_5C_6,\Big(\frac12\Big)^{\frac{1}{1-\sigma}}C_7\right\}.
\end{equation}

\textbf{Step 3: Blow-up result} Combining (\ref{3.15}) and (\ref{adddd3.20}), obviously, 
$
F^{\frac{1}{1-\sigma}}(t)\leq\frac{M_2}{M_1}F'(t),
$
which implies by Gronwall's inequality\[
F^\frac{\sigma}{1-\sigma}(t)\geq \frac{1}{F^{-\frac{\sigma}{1-\sigma}}(0)-\frac{M_2}{M_1}\frac{\sigma}{1-\sigma}t},\]  which yields $F(t)\rightarrow +\infty$ in finite time $T^*$ and 
\[T^*\leq F^{-\frac{\sigma}{1-\sigma}}(0)\frac{M_1}{M_2}\frac{1-\sigma}{\sigma}.\]
Here we fix some $\varepsilon>0$ such that 
\[F(0)=H^{1-\sigma}(0)+\varepsilon\Big(\int_\Omega u_1u_0dx+\frac12\|\nabla u_0\|_2^2\Big)>0.\]

In what follows, we will prove 
\[\lim\limits_{t\rightarrow T^{*-}}F(t)\rightarrow+\infty \Longrightarrow\lim\limits_{t\rightarrow T^{*-}}\|u\|_{p(x)}=+\infty.\] Let us consider the following three cases based on the definition of $F(t)$:\\
\textbf{Case 1}: $H(t)\rightarrow +\infty$.  
In this case, Lemma \ref{lem3.3} yields $\int_\Omega |u|^{p(x)}dx\rightarrow +\infty$. It easily follows Lemma \ref{lem2.1} that 
$\lim\limits_{t\rightarrow T^{*-}}\|u\|_{p(x)}=+\infty$.\\
\textbf{Case 2}: $\int_\Omega u_t udx\rightarrow +\infty$.  
Cauchy's inequality and the embedding $H_0^1(\Omega)\hookrightarrow L^{2}(\Omega)$  with the optimal constant $S>0$ illustrate 
\begin{equation}\label{3.22}
      \int_\Omega u_t udx  \leq \frac{1}{2}\|u_t\|_{2}^{2}+\frac{1}{2}\|u\|_{2}^{2}\leq \frac{1}{2}\|u_t\|_{2}^{2}+\frac{1}{2}S^{2}\|\nabla u\|_{2}^{2}.
\end{equation}
Recalling (\ref{3.2}) and $E(t)\leq E(0)<E_1$, we have
\begin{equation}\label{3.23}
    \begin{split}
       &\frac{1}{2}\|u_t\|_{2}^{2}+\frac{a}{2}\|\nabla u\|_2^2\leq\frac{1}{2}\|u_t\|_{2}^{2}+\frac{1}{2}\|\Delta u\|_2^2+\frac{a}{2}\|\nabla u\|_2^2+\frac{b}{2(\gamma+1)}\|\nabla u\|_2^{2(\gamma+1)}\\
       &\quad=E(t)+\int_\Omega\frac{1}{p(x)}|u|^{p(x)}dx\leq E(0)+\frac{1}{p^-}\int_\Omega|u|^{p(x)}dx.
    \end{split}
\end{equation}
It is direct by combining (\ref{3.22}) with (\ref{3.23}) and Lemma \ref{lem2.1} that if there exists  $\int_\Omega u_t udx\rightarrow +\infty$, then $\lim\limits_{t\rightarrow T^{*-}}\|u\|_{p(x)}=+\infty$.\\
\textbf{Case 3}: $\|\nabla u\|_2^2\rightarrow +\infty$. Here $\lim\limits_{t\rightarrow T^{-}*}\|u\|_{p(x)}=+\infty$ is clear due to (\ref{3.23}).

\end{proof}

\subsection{Blow-up for high initial energy}

In this section, we are committed to proving the finite time blow-up for high initial energy and  to estimating the upper bound of the blow-up time when the exponent $m(x)\equiv2$.  \begin{theorem}\label{thm4.1}
Let all assumptions in Lemma $\ref{lem2.6}$ be fulfilled.  Then the solution $u$ for problem \eqref{1.1} blows up in finite time.
 \end{theorem}
 \begin{proof}
 Let us prove this theorem by contradiction, i.e.  assume that the solution $u$ for problem \eqref{1.1} is global.

H\"{o}lder's inequality and \eqref{3.2} indicate that for all $t\in [0,\infty)$,
\begin{equation}\label{71}
\begin{split}
    &\|u\|_2=\Big\|u_0(x)+\int_0^tu_\tau d\tau\Big\|_2\leq \|u_0(x)\|_2+\int_0^t\|u_\tau\|_2 d\tau\\
    &\quad\leq\|u_0(x)\|_2+\sqrt{t}\Big(\int_0^t\|u_\tau\|_2^2 d\tau\Big)^{\frac 12}\quad\leq\|u_0(x)\|_2+\sqrt{t}(E(0)-E(t))^{\frac 12}.
\end{split}
\end{equation}
Since $u$ is a global solution of problem \eqref{1.1}, we have $E(t)\geq 0$ for all $t\in [0,\infty)$. Otherwise, there exists $t_0\in [0,\infty)$ such that $E(t_0)<0$. Choosing $u(x,t_0)$ as the new initial data, Theorem \ref{theo3.1} indicates that $u$ blows up in finite time, which is a contradiction. Thus,  \eqref{3.2} implies $0\leq E(t)\leq E(0)$. Further, \eqref{71} can be rewritten as
\begin{equation}\label{71}
    \|u\|_2\leq\|u_0(x)\|_2+\sqrt{t}(E(0))^{\frac 12}\quad \text{for }t\in [0,\infty).
\end{equation}

On the other hand,  \eqref{62} indicates  
\begin{equation}
\label{72}
\begin{split}
\frac{d}{dt}\|u\|_2^2=2\int_\Omega uu_tdx\ge 2\Psi(0)e^{Ct}+\frac{2p^-}{C}E(t)\ge 2\Psi(0)e^{Ct}>0.
\end{split}
\end{equation}
Integrating \eqref{72} from $0$ to $t$ yields
\begin{equation}
\label{73}
\begin{split}
    \|u\|_2^2&=\|u_0(x)\|_2^2+2\int_0^t\int_\Omega uu_\tau dxd\tau\geq\|u_0(x)\|_2^2+2\int_0^te^{C\tau}\Psi(0)d\tau\\
    &=\|u_0(x)\|_2^2+\frac{2}{C}(e^{Ct}-1)\Psi(0),
\end{split}
\end{equation}
 which contradicts \eqref{71} for $t$ sufficiently large. Thus, the solution $u$ for problem \eqref{1.1} blows up in finite time.
 \end{proof}
 \begin{theorem}\label{thm4.2}
Let all assumptions in Lemma $\ref{lem2.6}$ be fulfilled.   In addition, if 
\begin{equation}
\label{74}
E(0)\le \frac{C}{2p^-}\| u_0\|_2^2,
\end{equation}
then the solution $u$ for problem \eqref{1.1} blows up at some finite time $T^*$ in the sense of \[\lim_{t\to T^{*-}}\Big(\|u\|_2^2+\int_0^t (\| u\|_2^2+\|\nabla u\|_2^2)ds\Big)=\infty,\]
 and the upper of the blow-up time is given by 
\[T^*\le \frac{2(\|u_0\|_2^2+\rho\omega^2)}{(p^--2)\Big[\int_\Omega u_0u_1dx+\varrho\omega\Big]-2\|\nabla u_0\|^2_2},\]
where $\varrho=\frac{-2p^-E(0)+C\| u_0\|_2^2}{2p^-}$ and $\omega>0$  is sufficiently large such that 
\begin{equation}
\label{75}
(p^--2)\Big[\int_\Omega u_0u_1dx+\rho\omega\Big]-2\|\nabla u_0\|^2_2>0.
\end{equation}
 \end{theorem}
 \begin{proof}
 Obviously, Theorem \ref{thm4.1} implies that the solution $u$ for problem \eqref{1.1} blows up in finite time. Denote  by $T^*$ the blow-up time. Now, we need to estimate the upper of $T^*$.
 
 Define the auxiliary function
 \[\Upsilon(t)=\|u\|_2^2+\int_0^t(\|u\|_2^2+\|\nabla u\|_2^2)d\tau+(T^*-t)(\|u_0\|_2^2+\|\nabla u_0\|_2^2)+\varrho(t+\omega)^2\quad\text{for }t\in[0,T^*).\]
 By a direct computation, one has
 \begin{equation*}
\begin{split}
\Upsilon'(t)&=2\int_\Omega uu_tdx+\|u\|_2^2-\|u_0\|_2^2+\|\nabla u\|_2^2-\|\nabla u_0\|_2^2+2\varrho(t+\omega)\\
&=2\int_\Omega uu_tdx+2\int_0^t\int_\Omega (uu_\tau+\nabla u\nabla u_\tau)dxd\tau+2\varrho(t+\omega)\quad\text{for }t\in[0,T^*),
\end{split}
\end{equation*}
From the equality above and problem \eqref{1.1}, it is obtained that 
 \begin{equation*}
\begin{split}
\Upsilon''(t)&=2\|u_t\|_2^2+2\int_\Omega uu_{tt}dx+2\int_\Omega u u_tdx+2\int_\Omega \nabla u\nabla u_tdx+2\varrho\\
&=2\|u_t\|_2^2-2\|\Delta u\|_2^2-2a\|\nabla u\|_2^2-2b\|\nabla u\|_2^{2(\gamma+1)}+2\int_\Omega |u|^{p(x)}dx+2\varrho
\end{split}
\end{equation*}
for $t\in[0,T^*)$.
 Applying Cauchy-Schwarz inequality and Young's inequality, one has 
 \begin{equation*}
 \begin{split}
\xi(t):&=\Big[\|u\|_2^2+\int_0^t(\|u\|_2^2+\|\nabla u\|_2^2)d\tau+\varrho(t+\omega)^2\Big]\\
&\quad\times\Big[\|u_t\|_2^2+\int_0^t(\|u_\tau\|_2^2+\|\nabla u_\tau\|_2^2)d\tau+\varrho\Big]\\
&\quad-\Big[\int_\Omega uu_tdx+\int_0^t\int_\Omega (uu_\tau+\nabla u\nabla u_\tau) dxd\tau+\varrho(t+\omega)\Big]^2\ge 0
\end{split}
\end{equation*}
 for $t\in[0,T^*)$. Therefore, 
 \begin{equation}
\label{76}
\begin{split}
&\Upsilon(t)\Upsilon''(t)-\frac{p^-+2}{4}(\Upsilon'(t))^2\\
&\quad=\Upsilon(t)\Upsilon''(t)-\frac{p^-+2}{4}\Big[2\int_\Omega uu_tdx+2\int_0^t\int_\Omega (u u_\tau+\nabla u\nabla u_\tau)dxd\tau+2\varrho(t+\omega)\Big]^2\\
&\quad=\Upsilon(t)\Upsilon''(t)-(p^-+2)\Upsilon(t)\Big(\|u_t\|_2^2+\int_0^t(\| u_\tau\|_2^2+\|\nabla u_\tau\|_2^2)d\tau+\varrho\Big)+(p^-+2)\xi(t)\\
&\quad\quad+(p^-+2)(T^*-t)(\|u_0\|_2^2+\|\nabla u_0\|_2^2)\Big(\|u_t\|_2^2+\int_0^t(\|u_\tau\|_2^2+\nabla u_\tau\|_2^2)d\tau+\varrho\Big)\\
&\quad \ge \Upsilon(t)\eta(t)\quad\text{for }t\in[0,T^*),
\end{split}
\end{equation}
 where
 \begin{equation*}
\begin{split}
\eta(t)&=\Upsilon''(t)-(p^-+2)\Big(\|u_t\|_2^2+\int_0^t(\|u_\tau\|_2^2+\|\nabla u_\tau\|_2^2)d\tau+\varrho\Big)\\
&=-p^-\|u_t\|_2^2-2\|\Delta u\|_2^2-2a\|\nabla u\|_2^2-2b\|\nabla u\|_2^{2(\gamma+1)}+2\int_\Omega |u|^{p(x)}dx\\
&\quad-(p^-+2)\int_0^t(\|u_\tau\|_2^2+\|\nabla u_\tau\|_2^2)d\tau-p^-\varrho.
\end{split}
\end{equation*}
Using \eqref{3.1} and \eqref{3.2}, we obtain 
 \begin{equation}
\label{77}
\begin{split}
\eta(t)&=-2p^-E(t)+(p^--2)\|\Delta u\|_2^2+(p^--2)a\|\nabla u\|_2^2\\
&\quad+(p^--2(\gamma+1))b\|\nabla u\|_2^{2(\gamma+1)}-(p^-+2)\int_0^t(\|u_\tau\|_2^2+\|\nabla u_\tau\|_2^2)d\tau-p^-\varrho\\
&\ge-2p^-E(0)+(p^--2)a\|\nabla u\|_2^2+(p^--2)\int_0^t(\|u_\tau\|_2^2+\|\nabla u_\tau\|_2^2)d\tau-p^-\varrho\\
&\ge -2p^-E(0)++\frac{(p^--2)}{S^2}a\|u\|_2^2-p^-\varrho\quad \text{for }t\in[0,T^*).
\end{split}
\end{equation} 
Note that \eqref{72} implies 
\begin{equation}
\label{78}
\| u\|_2^2\ge \| u_0\|_2^2\quad \text{for }t\in[0,T^*).
\end{equation}
Thus, it follows by combining \eqref{76} with  \eqref{77} and \eqref{78} that for $t\in[0,T^*)$,
\begin{equation}
\label{79}
\begin{split}
\Upsilon(t)\Upsilon''(t)-\frac{p^-+2}{4}(\Upsilon'(t))^2 \ge \Upsilon(t)\Big\{-2p^-E(0)+C\| u_0\|_2^2-p^-\varrho\Big\}\ge 0
\end{split}
\end{equation}
here we use $\varrho=\frac{-2p^-E(0)+C\| u_0\|_2^2}{2p^-}$ and \eqref{74}.  
Noticing that 
\[\Upsilon(0)=\|u_0\|_2^2+T^*\|\nabla u_0\|_2^2+\varrho\omega^2>0\text{ and }\Upsilon'(0)=2\int_\Omega u_0u_1dx+2\varrho\omega>0,\]
thus, making use of Lemma \ref{lem2.5} yields that $\Upsilon(t)\to\infty$ as $t\to T^*$ with 
\[T^*\le \frac{4\Upsilon(0)}{(p^--2)\Upsilon'(0)}=\frac{2(\|u_0\|_2^2+T^*\|\nabla u_0\|_2^2+\rho\omega^2)}{(p^--2)\Big[\int_\Omega u_0u_1dx+\varrho\omega\Big]}.\]
It follows from \eqref{75} that
\[T^*\le \frac{2(\|u_0\|_2^2+\varrho\omega^2)}{(p^--2)\Big[\int_\Omega u_0u_1dx+\varrho\omega\Big]-2\|\nabla u_0\|^2_2}.\]
This proof is complete.
 \end{proof}

\subsection{A lower bound of the blow-up time}
In what follows, we may make full use of the strong damping $\Delta u_t$ to give a lower bound of the blow-up time. 
\begin{theorem}\label{thm4.3}
Let $N\ge 5.$ If all conditions of Theorem $\ref{theo3.1}$ or Theorem $\ref{thm4.2}$ are satisfied, and 
\[2<p^-\leq p(x)\leq p^+<2(N-1)/(N-4),\] 
then the lower bound for the blow-up time $T^{*}$ is given by
\begin{equation*}
\int_{R(0)}^{+\infty}\frac{1}{K_{1}y^{p^+-1}+K_2}dy\leq T^*,
\end{equation*}
where 
\begin{equation}\label{3.20}
    R(0)=\frac{1}{2}\|u_1\|_{2}^{2}+\frac{1}{2}\|\Delta u_0\|_2^2+\frac{a}{2}\|\nabla u_0\|_2^2+\frac{b}{2(\gamma+1)}\|\nabla u_0\|_2^{2(\gamma+1)},
\end{equation}
the constants $K_1$ and $K_2$ are defined in $(\ref{add828})$, $(\ref{1add828})$, respectively.
\end{theorem}
 \begin{proof} 
 Define an auxiliary function 
\begin{equation}\label{3.12}
\begin{split}
R(t)&=\frac{1}{2}\|u_t\|_{2}^{2}+\frac{1}{2}\|\Delta u\|_2^2+\frac{a}{2}\|\nabla u\|_2^2+\frac{b}{2(\gamma+1)}\|\nabla u\|_2^{2(\gamma+1)}\\
&=E(t)+\int_\Omega\frac{1}{p(x)}|u|^{p(x)}dx.
\end{split}
\end{equation}
by recalling \eqref{3.1}. Therefore, the conclusion of Theorem \ref{theo3.1} or Theorem $\ref{thm4.2}$ and $\eqref{Eq3}$ indicate $
\lim\limits_{t\rightarrow T^{*-}}R(t)=\infty.$
It follows from \eqref{3.2} that 
\begin{equation}\label{3.13}
\begin{split}
    R'(t)=E'(t)+\int_\Omega |u|^{p(x)-2}uu_tdx\le -\|\nabla u_t\|_2^2+\int_\Omega |u|^{p(x)-2}uu_tdx.
    \end{split}
\end{equation}
By using $\hbox{H\"{o}lder's}$ inequality, the embedding $H_0^1(\Omega)\hookrightarrow L^{2^*}(\Omega)$($2^*=\frac{2N}{N-2}$) with the optimal embedding constant $S_*$ and $\hbox{Young's}$ inequality, it follows that 
\begin{equation}
\begin{split}\label{3.31}
R^{\prime}(t)&\leq \Big\||u|^{p(x)-1}\Big\|_{\frac{2N}{N+2}}\|u_t\|_{2*}
-\|\nabla u_t\|_2^2\leq S_*\Big\||u|^{p(x)-1}\Big\|_{\frac{2N}{N+2}}\|\nabla u_t\|_{2}
-\|\nabla u_t\|_2^2\\
&\leq \frac{S_*^{2}}{2\varepsilon}
\||u|^{p(x)-1}\Big\|_{\frac{2N}{N+2}}^2+\frac{\varepsilon}{2} \|\nabla u_t\|_2^2-\|\nabla u_t\|_2^2\\
&\leq  \frac{S_*^{2}}{2\varepsilon}\Big[\Big(\int_{\{|u|\ge 1\}} |u|^{\frac{2N(p^+-1)}{N+2}}dx\Big)^{\frac{N+2}{N}}+\Big(\int_{\{|u|< 1\}} |u|^{\frac{2N(p^--1)}{N+2}}dx\Big)^{\frac{N+2}{N}}\Big]\\
&\leq  \frac{S_*^{2}}{4}\|u\|_{(p^+-1)\frac{2N}{N+2}}^{2(p^+-1)}+\frac{S_*^{2}}{4}|\Omega|^{\frac{N+2}{N}},
\end{split}
\end{equation}
where we choose $\varepsilon=2$ in last inequality. Noting that $\frac{2N(p^+-1)}{N+2}\leq \frac{2(N-1)}{N-4}$, then using the embedding $H_0^2(\Omega)\hookrightarrow L^{\frac{2N(p^+-1)}{N+2}}(\Omega)$ and \eqref{3.12}, one has
\begin{equation}
\begin{split}\label{add3.31}
R^{\prime}(t)&\leq  \frac{S_*^{2}}{4}\|u\|_{(p^+-1)\frac{2N}{N+2}}^{2(p^+-1)}+\frac{S_*^{2}}{4}|\Omega|^{\frac{N+2}{N}}
\le \frac{S_*^{2}}{4}B^{2(p^+-1)}\|\Delta u\|_{2}^{2(p^+-1)}+\frac{S_*^{2}}{4}|\Omega|^{\frac{N+2}{N}}\\
&\le \frac{S_*^{2}}{4}B^{2(p^+-1)}\Big(\frac{2}{a}\Big)^{p^+-1}\|\Delta u\|_{2}^{2(p^+-1)}+\frac{S_*^{2}}{4}|\Omega|^{\frac{N+2}{N}}
\\
&\le K_1(R(t))^{p^+-1}+K_2,
\end{split}
\end{equation}
here 
\begin{equation}
\label{add828}
K_1=\frac{S_*^{2}}{4}B^{2(p^+-1)}\Big(\frac{2}{a}\Big)^{p^+-1},\end{equation}
\begin{equation}
\label{1add828}
 K_2=\frac{S_*^{2}}{4}|\Omega|^{\frac{N+2}{N}}.
\end{equation}
Clearly, $(\ref{add3.31})$ implies
\begin{equation*}
\int_{R(0)}^{+\infty}\frac{1}{K_{1}y^{p^+-1}+K_2}dy\leq T^*.
\end{equation*}
This completes the proof of this Theorem.
\end{proof}

\section{Global existence and energy decay estimates}
In this section, we are committed to  showing the asymptotic stability. Let us first prove the global existence of solutions.
\begin{theorem}\label{theo5.2}
If  all conditions of Lemma $\ref{lem3.2}(2)$ hold, then the local solution $u$ of problem \eqref{1.1} is global.
\end{theorem}
\begin{proof}
It directly follows from (\ref{5.1}) (\ref{3.6}) and (\ref{3.1}) that
\begin{equation*}
\begin{split}
\int_\Omega\frac{1}{p(x)}|u|^{p(x)}dx
&\leq  \frac{1}{p^-}\max\Big\{B^{p^+}\|\Delta u\|_2^{p^+},B^{p^-}\|\Delta u\|_2^{p^-}\Big\}\\
&\leq  \frac{1}{p^-}\max\Big\{(B_1^2\|\Delta u\|_2^2)^{\frac{p^+-2}{2}},(B_1^2\|\Delta u\|_2^2)^{\frac{p^--2}{2}}\Big\}B_1^2\|\Delta u\|_2^2\\
&\leq   \frac{B_1^2}{p^-}\tilde{\alpha_2}^{\frac{p^--2}{2}}\left(2E(t)+2\int_\Omega\frac{1}{p(x)}|u|^{p(x)}dx\right),
\end{split}
\end{equation*}
which implies 
\begin{equation}
\label{5.2}
\int_\Omega\frac{1}{p(x)}|u|^{p(x)}dx
\leq 
\frac{{ 2B_1^{2}\tilde{\alpha_2}^{\frac{p^--2}{2}}}
}{p^-- { 2B_1^{2}\tilde{\alpha_2}^{\frac{p^--2}{2}}}}E(t).
\end{equation}
Further, (\ref{5.2}) and (\ref{3.1}) illustrate 
\begin{equation}
\begin{split}
\label{5.3}
&\frac{1}{2}\|u_t\|_{2}^{2}+\frac{1}{2}\|\Delta u\|_2^2+\frac{a}{2}\|\nabla u\|_2^2+\frac{b}{2(\gamma+1)}\|\nabla u\|_2^{2(\gamma+1)}\\
&\leq \frac{p^-}{p^-- { 2B_1^{2}\tilde{\alpha_2}^{\frac{p^--2}{2}}}}E(t)\leq  \frac{p^-}{p^-- 2}E(t)\leq  \frac{p^-}{p^-- 2}E(0),
\end{split}
\end{equation}
which implies the weak solution $u$ for problem (\ref{1.1}) exists globally.
\end{proof}

\begin{theorem}\label{theo5.4}
Under all conditions of Theorem  $\ref{Theo}$, suppose that  \[B_1^2\|\Delta u_0\|^2_2<\alpha_1\text{ and }0<E(0)<\tilde{E_2},\] 
then there exists a positive constant $K$ such that the energy functional satisfies
\begin{equation}
\label{5.4}
E(t)\leq E(0)e^{1-K t}
\end{equation}
 where $\tilde{E_2}=\Big(\frac{p^-}{p^+}\frac{1}{2}\Big)^{\frac{2}{p^--2}}\Big(\frac12-\frac{1}{p^+}\Big)\alpha_1^{\frac{p^-}{2}}\in(E(0),E_1)$, $K$  will be obtained later.
\end{theorem}
\begin{proof}
Obviously, Theorem \ref{theo5.2} implies that the solution $u$ is global. We borrow some ideas from \cite{MST2018,GHM2018,LGL2019}. Multiplying (\ref{1.1}) by $u$ and integrating over $\Omega\times(s,T)$  with $s<T$ yield 
\begin{equation*}
\begin{split}
&\int_s^T \frac{d}{dt}\Big[\int_\Omega uu_tdx\Big]dt+\int_s^T  \int_\Omega\nabla u\nabla u_tdxdt\\
&\quad+\int_s^T  \Big[\|\Delta u\|_2^2+a\|\nabla u\|_2^2+b\|\nabla u\|_2^{2(\gamma+1)}\Big]dt\\
&=\int_s^T  \Big[\int_\Omega |u|^{p(x)}dx+\|u_t\|_2^2-\int_\Omega |u_t|^{m(x)-2}u_t udx\Big]dt.
\end{split}\end{equation*}
By recalling (\ref{3.1}), and then combining the equality above, one has
\begin{equation*}
\begin{split}
&\int_s^T  \Big[E(t)-\frac{1}{2}\|u_t\|_{2}^{2}+\int_\Omega\frac{1}{p(x)}|u|^{p(x)}dx\Big]dt\\
&=\int_s^T  \Big[\frac{1}{2}\|\Delta u\|_2^2+\frac{a}{2}\|\nabla u\|_2^2+\frac{b}{2(\gamma+1)}\|\nabla u\|_2^{2(\gamma+1)}\Big]dt\\
&\leq\int_s^T  \Big[\|\Delta u\|_2^2+a\|\nabla u\|_2^2+b\|\nabla u\|_2^{2(\gamma+1)}\Big]dt\\
&=\int_s^T  \Big[\int_\Omega |u|^{p(x)}dx+\|u_t\|_2^2-\int_\Omega |u_t|^{m(x)-2}u_t udx\Big]dt\\
&\quad-\int_{s}^{T}\frac{d}{d t}\Big[ \int_{\Omega}uu_{t}dx\Big]dt -\int_s^T  \int_\Omega\nabla u\nabla u_tdxdt.
\end{split}
\end{equation*}
By simplifying the above inequality, we get
\begin{equation}
\label{5.5}
\begin{split}
\int_s^T E(t)dt&\leq \int_s^T  \int_\Omega\Big[1-\frac{1}{p(x)}\Big] |u|^{p(x)}dxdt+\frac32\int_s^T  \|u_t\|_2^2dxdt\\
&\quad-\int_s^T  \int_\Omega |u_t|^{m(x)-2}u_t udxdt-\int_{s}^{T}\frac{d}{d t}\Big[\int_{\Omega}uu_{t}dx\Big]dt
\\
&\quad-\int_s^T  \int_\Omega\nabla u\nabla u_tdxdt\\
&:=J_{1}+J_{2}+J_{3}+J_{4}+J_{5}
\end{split}
\end{equation}

\textbf{Step 1: Estimates for $\bm{J_1}$ and $\bm{J_2}$} Obviously, (\ref{5.2}) yields 
\begin{equation}
\label{5.6}
|J_1|\leq (p^+-1)
\frac{{ 2B_1^{2}\tilde{\alpha_2}^{\frac{p^--2}{2}}}
}{p^-- { 2B_1^{2}\tilde{\alpha_2}^{\frac{p^--2}{2}}}}\int_s^TE (t) dt.
\end{equation}
the embedding $H_0^1(\Omega)\hookrightarrow L^{2}(\Omega)$  with the optimal constant $S>0$ and (\ref{3.2}) induce 
\begin{equation}
\label{5.7}
\begin{split}
|J_2|&\leq \frac{3S^2}{2}\int_{s}^{T} \|\nabla u_t\|_2^2dt\leq-\frac{3S^2}{2}\int_{s}^{T} E'(t)dt\leq \frac{3S^2}{2} E(s).
\end{split}
\end{equation}

\textbf{Step 2: Estimates for $\bm{J_3}$} On the one hand, by using $\hbox{Young's}$ inequality with $0<\varepsilon_4<1$ and (\ref{3.2}), we obtain
\begin{equation}
\label{5.8}
\begin{split}
|J_{3}|&\leq \int_{s}^{T} \Big[\int_{\Omega}\Big(\varepsilon_4^{\frac{1}{1-m^-}}|u_{t}|^{m(x)}+\varepsilon_4|u|^{m(x)}\Big)dx\Big]dt\\
&\leq-\varepsilon_4^{\frac{1}{1-m^-}}\int_{s}^{T} E'(t)dt+\varepsilon_4\int_{s}^{T} \int_{\Omega}|u|^{m(x)}dxdt\\
&\leq\varepsilon_4^{\frac{1}{1-m^-}} \Big[E(s)-E(T)\Big]+\varepsilon_4\int_{s}^{T}\int_{\Omega}|u|^{m(x)}dxdt\\
&\leq\varepsilon_4^{\frac{1}{1-m^-}}E(s)+\varepsilon_4\int_{s}^{T}\int_{\Omega}|u|^{m(x)}dxdt.
\end{split}
\end{equation}
On the other hand, we apply Lemma \ref{lem2.1} and (\ref{5.3})  to have
\begin{equation}
\label{5.9}
\begin{split}
\int_{\Omega}|u|^{m(x)}dx&\leq \max\Big\{B_1^{m^-}\|\Delta u\|_2^{m^-},B_1^{m^+}\|\Delta u\|_2^{m^+}\Big\}\\&\leq  B^{m^-}_{1}\|\Delta u\|^{m^-}_2
\leq \Big(\frac{2p^-E(0)}{p^--2}\Big)^{\frac{m^--2}{2}}\frac{2p^-B^{m^-}_{1}}{p^-- 2}E(t).
\end{split}
\end{equation}
Therefore, we combine $(\ref{5.9})$ with $(\ref{5.8})$ to obtain
\begin{equation}
\label{5.10}
|J_{3}|\leq\varepsilon_4^{\frac{1}{1-m^-}} E(s)+\varepsilon_4\Big(\frac{2p^-E(0)}{p^--2}\Big)^{\frac{m^--2}{2}}\frac{2p^-B^{m^-}_{1}}{p^-- 2}\int_{s}^{T}E(t)dt,
\end{equation}

\textbf{Step 3: Estimates for $\bm{J_4}$ and $\bm{J_5}$} By applying Cauchy's inequality, (\ref{3.2}) and (\ref{5.3}), one shows
\begin{equation}
\label{5.11}
\begin{split}
|J_4|&=
\Big| \int_{\Omega}uu_t(x,s)dx- \int_{\Omega}uu_t(x,T)dx\Big|\\
&\leq  \frac{1}{2}\Big[\|u(x,s)\|_2^2+\|u_t(x,s)\|_2^2
+\|u(x,T)\|_2^2+\|u_t(x,T)\|_2^2
\Big]\\
&\leq  \frac{B^{2}_{1}}{2}\Big[\|\Delta u(x,s)\|_2^2+\|\Delta u(x,T)\|_2^2\Big]+\frac{1}{2}\Big[\|u_t(x,s)\|_2^2+\|u_t(x,T)\|_2^2
\Big]\\
&
\leq  \Big(\frac{2p^-B_1^{2}}{p^-- 2}+\frac{2p^-}
{p^-- 2}
\Big) E(s),
\end{split}
\end{equation}
\begin{equation}
\label{5.12}
\begin{split}
|J_5|&\leq   \int_s^T \|\nabla u\|_2\|\nabla u_t\|_2dt\\
&\leq \frac{\varepsilon_5}{2}\int_s^T \|\nabla u\|_2^2dt+\frac{1}{2\varepsilon_5}\int_s^T\|\nabla u_t\|_2^2dt\\
&\leq \frac{p^-\varepsilon_5}{p^--2}\int_s^T E(t)dt+\frac{1}{2\varepsilon_5}\int_s^T(-E'(t))dt\\
&= \frac{p^-\varepsilon_5}{p^--2}\int_s^T E(t)dt+\frac{1}{2\varepsilon_5}E(s)
\end{split}
\end{equation}
for $\varepsilon_5>0$.

Combining (\ref{5.5}) with (\ref{5.6}) (\ref{5.7}) and  (\ref{5.10})-(\ref{5.12}), it is clear that
\begin{equation}
\label{5.13}
\begin{split}
&\int_s^T E(t)dt\leq (p^+-1)
\frac{{ 2B_1^{2}\tilde{\alpha_2}^{\frac{p^--2}{2}}}
}{p^-- { 2B_1^{2}\tilde{\alpha_2}^{\frac{p^--2}{2}}}}\int_s^TE (t) dt
\\
&\quad+\Big[\frac{3S^2}{2}
+\varepsilon_4^{\frac{1}{1-m^-}}+\Big(\frac{2p^-B_1^{2}}{p^-- 2}+\frac{2p^-}
{p^-- 2}
\Big) +\frac{1}{2\varepsilon_5}\Big]E(s)\\
&\quad+\varepsilon_4\Big(\frac{2p^-E(0)}{p^--2}\Big)^{\frac{m^--2}{2}}\frac{2p^-B^{m^-}_{1}}{p^-- 2}\int_{s}^{T}E(t)dt+ \frac{p^-\varepsilon_5}{p^--2}\int_s^T E(t)dt.
\end{split}
\end{equation}
The condition $0<E(0)=G(\tilde{\alpha_2})<\tilde{E_2}=G\Big(\Big(\frac{p^-}{p^+}\frac{1}{2}\Big)^{\frac{2}{p^--2}}\alpha_1\Big)$ and the monotonicity of $G(\alpha)$ easily imply
\[\tilde{\alpha_2}<\Big(\frac{p^-}{p^+}\frac{1}{2}\Big)^{\frac{2}{p^--2}}\alpha_1<\alpha_1<1,\] 
which further illustrates
\[\delta:=(p^+-1)
\frac{{ 2B_1^{2}\tilde{\alpha_2}^{\frac{p^--2}{2}}}
}{p^-- { 2B_1^{2}\tilde{\alpha_2}^{\frac{p^--2}{2}}}}<1.\]
Choosing $0<\varepsilon_4<1$  and $0<\varepsilon_5<1$ sufficiently small such that
\[\varepsilon_4\Big(\frac{2p^-E(0)}{p^--2}\Big)^{\frac{m^--2}{2}}\frac{2p^-B^{m^-}_{1}}{p^-- 2}=\frac{1-\delta}{4}\text{ and }\frac{p^-\varepsilon_5}{p^--2}=\frac{1-\delta}{4}.\]
Therefore, (\ref{5.13}) can be rewritten as  \[\int_s^T E(t)dt\leq \frac {1}{K}E(s),\]
 where
 \begin{equation*}
 \begin{split}
 K&=1\Bigg/\Big[\frac{3S^2}{2}
+\varepsilon_4^{\frac{1}{1-m^-}}+\Big(\frac{2p^-B_1^{2}}{p^-- 2}+\frac{2p^-}
{p^-- 2}
\Big) +\frac{1}{2\varepsilon_5}\Big]\frac{2}{1-\delta}.
\end{split}
\end{equation*}
Let us make $T\to +\infty$, it follows that
\begin{equation}
\label{5.14}
\int_s^{+\infty} E(t)dt\leq\frac {1}{K}E(s).
\end{equation}
Lemma \ref{lem5.4} directly illustrates (\ref{5.4}).
\end{proof}

At the end of this paper, we give the result of asymptotic stability of  weak solutions. We define asymptotic stability of problem $(\ref{1.1})$ as follows: $u=0$ will be called  asymptotically stable(in the mean), if and only if
\[\lim\limits_{t\rightarrow\infty}E(t)=0,\text{ for all solutions }u \text{ to problem (\ref{1.1})}.\]
This notation was first presented by Pucci and Serrin \cite{PS1996}.
\begin{theorem}\label{theo5.5}
Suppose that  all conditions of Theorem $\ref{theo5.4}$ are satisfied, the rest field $u=0$ is  asymptotically stable.
\end{theorem}
\begin{proof} Recalling (\ref{5.3}),  it is so easy to verify  that
\begin{equation}\label{5.15}
\begin{split}
E(t)\geq \frac{p^--2}{p^-}\Big[\frac{1}{2}\|u_t\|_{2}^{2}+\frac{1}{2}\|\Delta u\|_2^2+\frac{a}{2}\|\nabla u\|_2^2+\frac{b}{2(\gamma+1)}\|\nabla u\|_2^{2(\gamma+1)}\Big]\geq 0.
\end{split}
\end{equation}
On the other hand, we apply the conclusion of Theorem \ref{theo5.4} to obtain 
$E(t)\leq 0\text{ as }t\to +\infty.$ Obviously, this theorem is true.
\end{proof}

\section*{Acknowledgements}
The author wishes to express her gratitude to Professor Wenjie Gao and Bin Guo in School of Mathematics, Jilin University for their support and constant encouragement. In particular, we thank Professor Baisheng Yan in Department of Mathematics,  Michigan State University for improving the quality of this paper.


\begin{thebibliography}{99}

\bibitem{AFP2021}
S. Antontsev,  J. Ferreira, E. Pi\c{s}kin, 
Existence and blow up of solutions for a strongly damped Petrovsky equation with variable-exponent nonlinearities,
\newblock {\em Electron. J. Differential Equations} \textbf{2021} (2021) 1--18.


\bibitem{CZ2009}
W. Y. Chen and Y. Zhou, Global nonexistence for a semilinear Petrovsky equation, \newblock {\em Nonlinear Anal.} \textbf{70} (2009) 3203--3208.

\bibitem{Fan1}
X. L. Fan and D. Zhao, On the spaces $L^{p(x)}(\Omega)$ and $W^{k,p(x)}(\Omega)$, \newblock {\em J. Math. Anal. Appl.} \textbf{263} (2001) 424--446.

\bibitem{Fan2}
X. L. Fan and Q. H. Zhang, Existence of solutions for $p(x)$-Laplacian Dirichlet problem, \newblock {\em Nonlinear Anal.} \textbf{52} (2003) 1843--1852.

\bibitem{G1998}
A. Guesmia, Existence globale et stabilisation interne non lin\'eaire d’un syst\`eme de Petrovsky, \newblock {\em Bull. Belg. Math. Soc.} \textbf{5} (1998) 583--594.

\bibitem{GHM2018}
S. Ghegal, I. Hamchi and S. A. Messaoudi, Global existence and stability of a nonlinear wave equation with variable-exponent nonlinearities,  \newblock {\em Appl. Anal.}  99:8 (2020)  1333--1343.

 \bibitem{GL2019}
B. Guo and X. L. Li, Bounds for the lifespan of solutions to fourth-order hyperbolic equations with initial data at arbitrary energy level, \newblock {\em Taiwanese J. Math.} \textbf{23} (2019) 1461--1477.

\bibitem{HL2020}
Y. Z. Han and Q. Li, Lifespan of solutions to a damped plate equation with logarithmic nonlinearity, \newblock {\em Evol. Equ. Control Theory.} doi:10.3934/eect.2020088


\bibitem{Kirchhoff}
G. Kirchhoff, Vorlesungen $\ddot{\mathrm{u}}$ber mathematische Physik, I, Teubner, Leipzig, 1883.

\bibitem{K2019}
J. R. Kang, Global nonexistence of solutions for von Karman equations with variable exponents, \newblock {\em Appl. Math. Lett.} \textbf{86} (2018) 249--255.

\bibitem{L1973}
H. A. Levine, Some nonexistence and instability theorems for solutions of formally parabolic equations of the form $Pu_t=-Au+F(u)$, \emph{Arch. Rational Mech. Anal.} \textbf{51} (1973) 371--386.

\bibitem{LSL2012}
G. Li, Y. N. Sun and W. J. Liu, Global existence and blow-up of solutions for a strongly damped Petrovsky system with nonlinear damping,
 \newblock {\em Appl. Anal.}  \textbf{91} (2012), no. 3,  575--586.
 
 \bibitem{LG2016}
F. S. Li and Q. Y. Gao, Blow-up of solution for a nonlinear Petrovsky type equation with memory, \newblock {\em Appl. Math. Comput.} \textbf{274} (2016) 383--392.

\bibitem{LG2017}
M. L. Liao, W. J. Gao, Blow-up phenomena for a nonlocal $p$-Laplace equation with Neumann boundary conditions, \emph{Arch. Math.} \textbf{108} (2017) 313--324.



\bibitem{L2021}
M. L. Liao, The lifespan of solutions for a viscoelastic wave equation with a strong damping and logarithmic nonlinearity, \newblock {\em Evol. Equ. Control Theory.} doi:10.3934/eect.2021025

\bibitem{LGZ2020}
M. L. Liao, B. Guo, X.Y. Zhu, Bounds for blow-up time to a viscoelastic hyperbolic equation of Kirchhoff type with variable sources,  \newblock {\em Acta Appl.  Math.}  \textbf{170} (2020) 755--772 .





\bibitem{LSW2019}
L. H. Liu, F. L. Sun and Y. H. Wu, Blow-up of solutions for a nonlinear Petrovsky type equation with initial data at arbitrary high energy level, \newblock {\em Bound. Value Probl.} (2019) 2019:15
 
 
 
 \bibitem{LGL2019}
X. L. Li, B. Guo and M. L. Liao,  Asymptotic stability of solutions to quasilinear hyperbolic equations with variable sources, \newblock {\em Comput. Math. Appl.} \textbf{79} (2020) 1012--1022.

\bibitem{M1999}
P. Martinez, A new method to obtain decay rate estimates for dissipative systems,  \newblock {\em ESAIM: Control Optim. Calc. Var.} \textbf{4} (1999) 419--444.

\bibitem{Messaoudi2002}
S. A. Messaoudi, Global existence and nonexistence in a system of Petrovsky, \newblock {\em J. Math. Anal. Appl.}  \textbf{265} (2002) 296--308.

\bibitem{SA2017}
S. A. Messaoudi, A. A. Talahmeh and J. H. Al‐-Smail, 
Nonlinear damped wave equation: Existence and blow-up, \newblock {\em Comput. Math. Appl.} \textbf{74} (2017) 3024--3041.

\bibitem{MST2018}
S. A. Messaoudi, J. H. Al--Smail and A. A.Talahmeh, Decay for solutions of a nonlinear damped wave equation with variable-exponent nonlinearities,\newblock {\em Comput. Math. Appl.} \textbf{76} (2018) 1863--1875.

\bibitem{PS1996}
P. Pucci and J. Serrin, Asymptotic stability for nonlinear parabolic systems, in: Energy Methods in Continuum Mechanics, Springer, The Netherlands, (1996) 66--74.

 
\bibitem{SLW2019}
 F. L. Sun, L. S.  Liu, Y. H. Wu, Global existence and finite time blow-up of solutions for the semilinear pseudo-parabolic equation with a memory term,
\emph{Appl. Anal.} \textbf{98} (2019)  735--755.

\bibitem{TS2012}
F. Tahamtani and M. Shahrouzi, Existence and blow up of solutions to a Petrovsky equation with memory and nonlinear source term, \newblock {\em Bound. Value Probl.} \textbf{2012}, 50 (2012).

\bibitem{Wu1}
S. T. Wu and L. Y. Tsai, On global existence and blow-up of solutions for an integro-differential equation with strong damping, \newblock {\em Taiwan. J. Math.} \textbf{10} (2006), no. 4, 979--1014.

\bibitem{WT2009}
S. T. Wu and L. Y. Tsai, On global solutions and blow-up of solutions for a nonlinearly damped Petrovsky system, \newblock {\em Taiwanese J. Math.}  \textbf{13} (2009), no. 2A,  545--558.

\bibitem{Wu2}
S. T. Wu and L. Y. Tsai, Blow-up of positive-initial-energy solutions for an integro-differential equation with nonlinear damping, \newblock {\em Taiwan. J. Math.} \textbf{14} (2010), no. 5, 2043--2058.

 \bibitem{WU2018}
S. T. Wu, Lower and upper bounds for the blow-up time of a class of damped fourth-order nonlinear evolution equations, \newblock {\em J. Dyn. Control Syst.} \textbf{24} (2018), no. 2, 287--295.

\bibitem{Yang}
Z. F. Yang and Z. G. Gong, Blow-up of solutions for viscoelastic equations of Kirchhoff type with arbitrary positive initial energy, \newblock {\em Electron. J. Differential Equations} \textbf{2016} (2016), no. 332, 1--8.

\bibitem{Zhou2015}
J. Zhou, Global existence and blow-up of solutions for a Kirchhoff type plate equation with damping, \newblock {\em Appl. Math. Comput.} \textbf{265} (2015) 807--818.

\end{thebibliography}
\end{document}